\documentclass[12pt]{article} 

\usepackage{latexsym,amsmath,amssymb,a4wide,amscd}
\usepackage[latin1]{inputenc}
\usepackage[T1]{fontenc}

\usepackage{array}

\usepackage{amsthm}

\usepackage{url}

\usepackage{mathabx}

\usepackage{longtable}

\usepackage{amsthm}
\numberwithin{equation}{section}
\numberwithin{figure}{section}
\theoremstyle{plain}
\newtheorem {satz}{Theorem}[section]
\newtheorem {lemma}[satz]{Lemma}
\newtheorem {prop}[satz]{Proposition}

\theoremstyle{definition}
\newtheorem {deff}[satz]{Definition}
\theoremstyle{remark}
\newtheorem {remark}[satz]{Remark}

\usepackage[pdftex]{graphicx} 

\usepackage[nooneline]{caption}

\begin{document}

\date{}

\title{\Large {\bf Partitioning the triangles of the cross polytope into surfaces}}

\author{Jonathan Spreer}

\maketitle

\subsection*{\centering Abstract}

{\em
	We present a constructive proof that there exists a decomposition of the $2$-skeleton of the $k$-dimensional cross polytope $\beta^k$ into closed surfaces of genus $g \leq 1$, each with a transitive automorphism group given by the vertex transitive $\mathbb{Z}_{2k}$-action on $\beta^k$. Furthermore we show that for each $k \equiv 1,5(6)$ the $2$-skeleton of the $(k-1)$-simplex is a union of highly symmetric tori and M\"obius strips.
}\\
\\
\textbf{MSC 2010: } {\bf 52B12}; 
52B70; 
57Q15;  
57M20; 
05C10; 
\\
\textbf{Keywords: } cross polytope, simplicial complexes, triangulated surfaces, difference cycles.

\section{Introduction}

It is an obvious fact that the $(k-1)$-simplex $\Delta^{k-1}$ contains all triangulations on $k$ vertices (possibly after a relabeling of the vertices). Hence, its {\em $i$-skeleton} $\operatorname{skel}_i(\Delta^{k-1})$, i.\ e.\ the set of all $i$-dimensional faces of $\Delta^{k-1}$, $i < k$, can be seen as the space of all triangulated $i$-manifolds with $m \leq k$ vertices.

One way of analyzing this space is to create partitions of $\operatorname{skel}_i(\Delta^{k-1})$ where every part fulfills certain conditions.

Of course, one could divide $\operatorname{skel}_i(\Delta^{k-1})$ into ${k \choose i+1}$ parts, each containing one $i$-simplex which is a topological ball and hence fulfills the constraint to be a bounded manifold. However, if we want to partition $\operatorname{skel}_i(\Delta^{k-1})$ into fewer parts, with the additional condition that, for example, each part is a bounded manifold, more restrictions apply. As a consequence, the partition becomes more meaningful regarding the structure of $\Delta^{k-1}$.

\medskip
In the centrally symmetric case, the $k$-dimensional {\em cross polytope}, i.\ e.\ the convex hull of $2k$ points $x_i^{\pm} = (0, \ldots , 0, \pm 1, 0, \ldots , 0) \in \mathbb{R}^k$, $1 \leq i \leq k$, is the space of all centrally symmetric triangulations, i.\ e.\  all triangulations having a symmetry of order $2$ without fixed points. Thus, partitions of the $i$-skeleton of $\beta^k$ which follow suitable constrains could give new insights into the class of centrally symmetric triangulations which forms an interesting family of triangulations, previously investigated by Gr\"unbaum \cite{Gruenbaum03ConvPoly}, Jockush \cite{Jockusch95InfFamNearNeighCentSymmS3}, K\"uhnel \cite{Kuhnel96CentrSymmTightSurf}, Effenberger and K\"uhnel \cite{Effenberger08CentSymmS2S27}, Lassmann and Sparla \cite{Lassmann00ClassCentSymmCycS2S2} and many others.

\medskip
However, a more general and systematic approach to this method of analyzing a polytope is to consider partitions of the $i$-skeleton of arbitrary (simplicial) polytopes, ordered by the size of the partition, say $l$. In the special case $l=2$, i.\ e.\ where the $i$-skeleton $\operatorname{skel}_{i} (P)$, $1 \leq i \leq (k-2)$, of a $k$-polytope $P$ can be partitioned into two (possibly bounded) PL $i$-manifolds $M_1$ and $M_2$ with $ M_1 \cup M_2 = \operatorname{skel}_{i} (P)$ such that $ M_1 \cap M_2 \subset \operatorname{skel}_{i-1} (P)$, the polytope $P$ is called {\em decomposable}.

For $i=1$, finding decomposable polytopes $P$ means decomposing the edge graph of $P$ into two (connected) graphs of degree at most $2$. This problem was solved by Gr\"unbaum and Malkevitch \cite{Grunbaum76PairsEdgeDisjHamCircs} as well as Martin \cite{Martin76CyclesHamiltoniens}.  

In the case $i=2$, the question was investigated by Betke, Schulz and Wills (see \cite{Betke76ZerlegSkel}) with the result that there exist only five polytopes which allow a decomposition of the set of triangles into two surfaces with boundary, namely the $4$-simplex, the $5$-simplex, the $4$-dimensional cross polytope, the $6$-vertex cyclic $4$-polytope and the double pyramid over the $3$-simplex. 

\medskip
It is fairly easy to agree that in general decompositions of a polytope $P$ with $l$ small are very restrictive towards the local combinatorial structure of $P$: Just counting the number of triangles sharing one common edge of $\beta^k$ leads to the conclusion that the $2$-skeleton of $\beta^k$ cannot be partitioned into less than $l = k-2$ surfaces.

On the other hand, a decomposition with not too many parts could be enlightening towards a deeper understanding of how skeletons of high co-dimension in a simplicial polytope are structured locally. For surfaces in the $k$-dimensional cross polytope this can be demonstrated by our main result.

\begin{satz}
\label{thm:main}
The $2$-skeleton of the $k$-dimensional cross polytope $\beta^k$ 
can be decomposed into triangulated vertex transitive closed surfaces.

More precisely, if $k \equiv 1,2 \, (3)$, $\operatorname{skel}_2 ( \beta^k )$ decomposes into $\frac{(k-1)(k-2)}{3}$ triangulated vertex transitive closed surfaces of Euler characteristic $0$ on $2k$ vertices and, if $k \equiv 0 \, (3)$, into $\frac{k}{3}$ disjoint copies of $\partial \beta^3$ (on $6$ vertices each) and $\frac{k(k-3)}{3}$ triangulated vertex transitive closed surfaces of Euler characteristic $0$ on $2k$ vertices.
\end{satz}

The case $k=4$ equals one of the five decompositions already described by Betke, Schulz and Wills in \cite{Betke76ZerlegSkel}. It is worthwhile mentioning that Theorem \ref{thm:main} defines a decomposition of $\operatorname{skel}_2 ( \beta^k )$ not only into triangulated surfaces but into $0$-Hamiltonian closed surfaces with transitive automorphism group, i.\ e.\ each part of the partition contains $2k$ vertices and has an automorphism group of order at least $4k$ (a cyclic part with $2k$ vertices plus the centrally symmetric element). Hence, the space of all centrally symmetric triangulated surfaces - which is a highly symmetric space - can be partitioned into relatively large pieces of such surfaces while these surfaces in some cases consist of surprisingly many 
different combinatorial types (see Table \ref{tab:combTypes}). In fact, it follows from the construction of the partition that each centrally symmetric triangulated torus or Klein bottle with cyclic $\mathbb{Z}_{2k}$ symmetry occurs in this partition as a proper part.

\medskip
Under this point of view it is surprising that a similar partition of the $2$-skeleton of $\Delta^{k-1}$ fails to exist. Instead, it seems that only a partition into partly bounded triangulated surfaces is possible by virtue of the following theorem.

\begin{satz}
	\label{thm:deltak}
	Let $k>1$, $k \equiv 1,5(6)$. Then the $2$-skeleton of $\Delta^{k-1}$ decomposes into $\frac{k-1}{2}$ collections of M\"obius strips
	$$ M_{l,k} := \{ ( l : l : k - 2l ) \}, $$
	$1 \leq l \leq \frac{k-1}{2}$ each with $n:=\operatorname{gcd} (l,k)$ isomorphic connected components on $\frac{k}{n}$ vertices and $\frac{k^2 -6k +5}{12}$ collections of tori
	$$ S_{l,j,k} := \{ ( l : j : k - l - j ),( l : k-l-j : j ) \}, $$
	$ 1 \leq l < j < k - l - j$, with $m:=\operatorname{gcd} (l,j,k)$ connected components on $\frac{k}{m}$ vertices each.
\end{satz}

In order to proof our main results, we will first introduce some terms and methods to work with triangulations with transitive cyclic symmetry (cf.\ Section \ref{sec:cyclic}) before we will establish some useful lemma and a proof for  Theorem \ref{thm:main} (cf. Section \ref{sec:betak}) and Theorem \ref{thm:deltak} (cf. Section \ref{sec:deltak}).

\section{Simplicial complexes with cyclic symmetry}
\label{sec:cyclic}

Highly symmetric simplicial complexes $C$ allow a very efficient description by the generators of its automorphism group $\operatorname{Aut} (C)$ together with a system of orbit representatives of the complex under the action of $\operatorname{Aut} (C)$. In this way, all kinds of statements involving highly symmetric simplicial complexes often have elegant and easy to handle proofs.

In particular, this is true for complexes with transitive automorphism group, sometimes just called {\em transitive complexes}: As the automorphism group acts transitive on the set of vertices, all vertex links of a simplicial complex are {\em combinatorially isomorphic}, i.\ e.\ they are equal up to a relabeling of their vertices. Hence, a neighborhood of each vertex determines the complex globally and, as a consequence, global properties can be calculated locally.

\medskip
An important class of transitive complexes are the ones which contain a cyclic automorphism acting regular on the set of vertices. 

By a relabeling of the vertices to an integer labeling $0 , \ldots , n$, the cyclic group action can be represented by the vertex shift $v \mapsto v+1 \mod n$. In this way, the simplices of a complex can be divided into equivalence classes by simply calculating the differences between the vertex labels of each simplex and each of these equivalence classes defines a cyclic orbit. More precisely we have the following definition.

\begin{deff}[Difference cycle]
	\label{def:diffCycles}
	Let $a_i \in \mathbb{N} \setminus \{ 0 \}$, $0 \leq i \leq d$, $ n := \sum_{i=0}^{d} a_i$ and $\mathbb{Z}_n = \langle (0,1, \ldots , n-1) \rangle$. The simplicial complex
	$$ ( a_0 : \ldots : a_{d} ) := \mathbb{Z}_n \cdot \{0 , a_0 , \ldots , \Sigma_{i=0}^{d-1} a_i \}, $$
	where $\cdot$ is the induced cyclic $\mathbb{Z}_n$-action on subsets of $\mathbb{Z}_n$, is called {\em difference cycle of dimension $d$ on $n$ vertices}. The number of its elements is referred to as the {\em length} of the difference cycle. If a complex $C$ is a union of difference cycles of dimension $d$ on $n$ vertices and $\lambda$ is a unit of $\mathbb{Z}_n$ such that the complex $\lambda C$ (obtained by multiplying all vertex labels modulo $n$ by $\lambda$) equals $C$, then $\lambda$ is called a {\em multiplier} of $C$.
\end{deff}

Note that for any unit $\lambda \in \mathbb{Z}_n^{\times}$, the complex $\lambda C$ is combinatorially isomorphic to $C$. In particular, all $\lambda \in \mathbb{Z}_n^{\times}$ are multipliers of the complex $\bigcup_{\lambda \in \mathbb{Z}_n^{\times}} \lambda C$ by construction. 

The definition of a difference cycle above is similar to the one given in \cite{Kuehnel96PermDiffCyc}. For a more thorough introduction into the field of the more general difference sets and their multipliers see Chapter VI and VII in \cite{Beth99DesignTheory}.

Throughout this article we will look at difference cycles as simplicial complexes with a transitive automorphism group given by the cyclic $\mathbb{Z}_n$-action on its elements: Every $(d+1)$-tuple $\{ x_0 , \ldots , x_{d} \}$ is interpreted as a $d$-simplex $\Delta^d = \langle x_0, \ldots , x_{d} \rangle$.  A simplicial complex $C$ is called {\em transitive}, if its group of automorphisms acts transitively on the set of vertices. In particular, any union of difference cycles is a transitive simplicial complex.

\begin{remark}
	It follows from Definition \ref{def:diffCycles} that the set of difference cycles of dimension $d$ on $k$ vertices defines a partition of the $d$-skeleton of the $(k-1)$-simplex. Two $(d+1)$-tuples $(a_0 , \ldots , a_{d})$ and $(b_0 , \ldots , b_{d})$ with $\Sigma_{i=0}^{d} a_i = \Sigma_{i=0}^{d} b_i = k$ define the same difference cycle if and only if for a fixed $j \in \mathbb{Z}$ we have $a_{(i + j) \! \mod (d+1)} = b_i$ for all $0 \leq i \leq d$. 
\end{remark}

\begin{prop}
	\label{prop:lengthOfDiffcycles}
	Let $( a_0 : \ldots : a_{d} )$ be a difference cycle of dimension $d$ on $n$ vertices and $1 \leq k \leq d+1$ the smallest integer such that $k \mid (d+1)$ and $a_i = a_{i+k}$, $0 \leq i \leq d-k$. Then $( a_0 : \ldots : a_{d} )$ is of length $\sum_{i=0}^{k-1} a_i = \frac{nk}{d+1}.$
\end{prop}

\begin{proof}
	We set $m:=\frac{nk}{d+1}$ and compute
	\small
	\begin{eqnarray}
		\left \langle 0 + m, a_0 + m, \ldots , (\Sigma_{i=0}^{d-1} a_i) + m \right \rangle &=& \left \langle \Sigma_{i=0}^{k-1} a_i , \Sigma_{i=0}^{k} a_i, \ldots , \Sigma_{i=0}^{d-1} a_i , 0 , a_1 , \ldots , \Sigma_{i=0}^{k-2} a_i \right \rangle \nonumber \\
	 &=& \left \langle 0 , a_0, \ldots , \Sigma_{i=0}^{d-1} a_i \right \rangle \nonumber
	\end{eqnarray}
	\normalsize
	(all entries are computed modulo $n$). Hence, for the length $l$ of $( a_0 : \ldots : a_{d} )$ we have $l \leq \frac{nk}{d+1}$ and since $k$ is minimal with $k \mid (d+1)$ and $a_i = a_{i+k}$, the upper bound is attained.
\end{proof}

\section{The decomposition of $\operatorname{skel}_2 (\beta^k)$ into closed cyclic surfaces}
\label{sec:betak}

In this section we will prove our main result Theorem \ref{thm:main}. 

\medskip
First of all let us mention that it will turn out to be very convenient to look at the boundary of the $k$-dimensional cross polytope in terms of the abstract simplicial complex

\begin{equation}
	\partial \beta^k = \{ \langle a_1 , \ldots , a_k \rangle \, | \, a_i \in \{ 0, \ldots , 2k-1 \}, \{ i , k+i \} \not \subseteq \{ a_1 , \ldots , a_k \},\, \forall \, 0 \leq i \leq k-1 \}.
\end{equation}

In particular, the diagonals of $\beta^k$ are precisely the edges $\langle i , k+i \rangle$, $1 \leq i \leq k$, and thus coincide with the difference cycle $(k\!:\!k)$. 

The proof of Theorem \ref{thm:main} itself will consist of an explicit construction of pairwise disjoint cyclic closed surfaces in the $2$-skeleton of the cross polytope $\beta^k$ as well as a proof of their topological types for any given integer $k \geq 3$.

However, let us first state a number of lemmata which will be helpful in the following.

\begin{lemma}
\label{lem:partBetaK}
The $2$-skeleton of $\beta^k$ can be written as the following set of difference cycles:
$$ (l : j : 2k-l-j) , (l : 2k-l-j : j) $$
for $0 < l < j < 2k-l-j$, $k \not \in \{ l,j,l + j \}$, and
$$ (j : j : 2(k-j)) $$
for $0 < j < k $ with $2j \neq k$. If $k \not\equiv 0 \, (3)$ all of them are of length $2k$, if $k \equiv 0\, (3)$ the difference cycle $(\frac{2k}{3} : \frac{2k}{3} : \frac{2k}{3}) $ has length $\frac{2k}{3}$.
\end{lemma}

\begin{proof}
	Let $\beta^k$ be the $k$-dimensional cross polytope with vertices $\{ 0 , \ldots , 2k-1 \}$ and diagonals $\{ j, k+j \}$, $0 \leq j \leq k-1$ . It follows from the recursive construction of $\beta^k$ as the double pyramid over $\beta^{k-1}$ that it contains all $3$-tuples of vertices as triangles except the ones including a diagonal. Thus, a difference cycle of the form $(a : b : c )$ lies in $\operatorname{skel}_2 (\beta^k)$ if and only if $k \not \in \{ a,b,a+b \}$ . In particular, $\operatorname{skel}_2 (\beta^k)$ is a union of difference cycles.
	
	Note that each ordered $3$-tuple $0 < l < j < 2k - l - j$ defines exactly two distinct difference cycles on the set of $2k$ vertices, namely
	$$ (l : j : 2k-l-j) \textrm{ and } (l : 2k-l-j : j) $$
	and it follows immediately that there is no other difference cycle $(a : b : c)$, $ k \notin \{ a,b,a+b \}$ on $2k$ vertices with $a,b,c$ pairwise distinct.
	
	For any positive integer $0 < j < k$ with $2j \neq k$ there is exactly one difference cycle
	$$ (j : j : 2k-2j)), $$
	and since $j$ must fulfill $0 < 2j < 2k$, there are no further difference cycles without diagonals with at most two different entries.
	
	The length of the difference cycles follows directly from Proposition \ref{prop:lengthOfDiffcycles} with $d=2$ and $n=2k$.
\end{proof}

\begin{lemma}
	\label{lem:euler}
	A closed $2$-dimensional pseudomanifold $S$ defined by $m$ difference cycles of full length on the set of $n$ vertices has Euler characteristic $\chi (S) = (1-\frac{m}{2})n$.
\end{lemma}

\begin{proof}
	Since all difference cycles are of full length, $S$ consists of $n$ vertices and $m \cdot n$ triangles. Additionally, the pseudo manifold property asserts that $S$ has $\frac{3}{2} m \cdot n$ edges and thus 
	$$  \chi (S) = n - \frac{3}{2} m \cdot n + m \cdot n = n(1 - \frac{m}{2}).$$
\end{proof}

\begin{lemma}
	\label{lem:typeA}
	Let $0 < l < j < 2k-l-j$, $k \not \in \{ l,j,l+j \}$ and $m:=\operatorname{gcd}(l,j,2k)$. Then
	$$ S_{l,j,2k} := \{ (l : j : 2k-l-j) , (l : 2k-l-j : j) \} \cong \{ 1 , \ldots , m \} \times \mathbb{T}^2, $$ 
	where all connected components of $S_{l,j,2k}$ are combinatorially isomorphic to each other.
\end{lemma}

\begin{proof}
	The link of vertex $0$ in $S_{l,j,2k}$ is equal to the cycle 
	
	\begin{center}
		\includegraphics[width=0.38\textwidth]{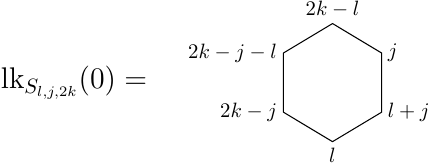}
	\end{center}
	
	Since $ 0 < l < j < 2k-l-j$ and $k \not \in \{ l,j,l+j \}$, all vertices are distinct and $\operatorname{lk}_{S_{l,j,2k}} (0)$ is the boundary of a hexagon. By the vertex transitivity all other links are also hexagons and $S_{l,j,2k}$ is a surface.
	
	Since $l$, $j$ and $2k-l-j$ are pairwise distinct, both $(l : j : 2k-l-j)$ and $(l : 2k-l-j : j)$ have full length and by Lemma \ref{lem:euler} the surface has Euler characteristic $0$.  
	
	In order to see that $S_{l,j,2k}$ is oriented, we look at the (oriented) boundary of the triangles in $S_{l,j,2k}$ in terms of $1$-dimensional difference cycles:
	\begin{eqnarray}
	\partial ( l : j : 2k - l - j ) &=& (j : 2k - j) - (l+j : 2k - l - j) + (l : 2k - l) \nonumber \\
	\partial ( l : 2k - l - j : j ) &=& (2k - l - j : l+j) - (2k-j : j) + (l : 2k - l) \nonumber \\
	&=& (j : 2k - j) - (l+j : 2k - l - j) + (l : 2k - l) \nonumber 
	\end{eqnarray}
	and thus $\partial ( l : j : 2k - l - j ) - \partial ( l : 2k - l - j : j ) = 0$ and $S_{l,j,2k}$ is oriented.
	
	Now consider
	$$ ( l : j : 2k-l-j ) = \mathbb{Z}_{2k} \cdot \langle 0 , l , l+j \rangle $$
	Clearly, $\langle (0+i)\! \mod 2k, (l+i)\! \mod 2k, (l+j+i)\! \mod 2k \rangle$ share at least one vertex if $i \in \{0,l,2k-l,j,2k-j,2k-l-j,l+j \}$. For any other value of $i < 2k$, the intersection of the triangles is empty. By iteration it follows that $( l : j : 2k-l-j ) $ has exactly $\operatorname{gcd}(0,l,2k-l,j,2k-j,2k-l-j,l+j) = \operatorname{gcd}(l,j,2k) = m$ connected components. The same holds for $( l : 2k-l-j : j ) $. The complex $(0, \ldots , (2k-1))^i \cdot \left \langle 0, l, l+j \right \rangle$ is disjoint to $\left \langle 0, l, 2k-j \right \rangle$ for $i \notin \{0,l,2k-l,j,2k-j,2k-l-j,l+j \}$. Together with the fact that $\operatorname{star}_{S_{l,j,2k}} (0)$ consists of triangles of both $( l : j : 2k-l-j )$ and $( l : 2k-l-j : j )$ it follows that $S_{l,j,2k}$ has $m$ connected components and by a shift of the indices one can see that all of them must be combinatorially isomorphic. Altogether it follows that $S_{l,j,2k} \cong \{ 1 , \ldots , m \} \times \mathbb{T}^2$.
\end{proof}

\begin{remark}
Some of the connected components of the surfaces presented above are combinatorially isomorphic to the so-called Altshuler tori
$$ \{ (1 : n-3 : 2) , (1 : 2 : n-3) \} $$
with $n = \frac{2k}{m} \geq 7$ vertices (cf. proof of Theorem $4$ in \cite{Altshuler71PolyhedralRealizationsTori}). However, other triangulations of transitive tori are part of the decomposition as well: in the case $k=6$, there are four different combinatorial types of tori. This is in fact the total number of combinatorial types of transitive tori on $12$ vertices (cf. Table \ref{tab:listBeta}). The number of distinct combinatorial types of centrally symmetric transitive tori for $k\leq 30$ is listed in Table \ref{tab:combTypes}.
\begin{figure}
 \begin{center}
	\begin{tabular}{|l|c||l|c||l|c||l|c|}
		\hline
		$k$ & $\#$ comb. types & $k$ & $\#$ comb. types & $k$ & $\#$ comb. types & $k$ & $\#$ comb. types \\
		\hline
		\hline
		$3$ & $0$ & $4$ & $1$ & $5$ & $1$ & $6$ & $4$ \\
		\hline
		$7$ & $2$ & $8$ & $3$ & $9$ & $4$ & $10$ & $6$ \\
		\hline
		$11$ & $4$ & $12$ & $9$ & $13$ & $5$ & $14$ & $8$ \\
		\hline
		$15$ & $11$ & $16$ & $7$ & $17$ & $7$ & $18$ & $12$ \\
		\hline
		$19$ & $8$ & $20$ & $13$ & $21$ & $15$ & $22$ & $12$ \\
		\hline
		$23$ & $10$ & $24$ & $17$ & $25$ & $13$ & $26$ & $14$ \\
		\hline
		$27$ & $16$ & $28$ & $17$ & $29$ & $13$ & $30$ & $26$ \\
		\hline
	\end{tabular}
 \end{center}
 \caption{Number of combinatorially distinct types of centrally symmetric transitive tori in $\beta^k$, $k\leq 30$. \label{tab:combTypes}}
\end{figure}
\bigskip
\end{remark}

\begin{remark}
	All centrally symmetric transitive surfaces ({\it cst-surfaces} for short) $S_{l,j,2k}$ from Lemma \ref{lem:typeA} can be constructed using the function \texttt{SCSeriesCSTSurface(l,j,2k)} from the \texttt{GAP}-package \textsf{simpcomp} \cite{simpcomp,simpcompISSAC}, maintained by Effenberger and the author. If the second parameter is not provided (\texttt{SCSeriesCSTSurface(l,2k)}), the surface $S_{l,2k}$ from Lemma \ref{lem:typeB} is generated. 
\end{remark}

\begin{lemma}
	\label{lem:BSeriesPartition}
	Let
	\small
	\begin{eqnarray}
		M &:=& \bigg\{  (j : j : 2(k-j)) \mid 0 < j < k ; 2j \neq k \bigg\}, \nonumber \\
		M_1 &:=& \left \{  (l : l : 2(k-l)) \mid 1 \leq l \leq \left \lfloor \frac{k-1}{2} \right \rfloor \right \} \textrm{ and} \nonumber \\
		M_2 &:=& \left \{  (k-l : k-l : 2l)) \mid 1 \leq l \leq \left \lfloor \frac{k-1}{2} \right \rfloor \right \}. \nonumber
	\end{eqnarray}
	\normalsize
	For all $k \geq 3$ the triple $(M,M_1,M_2)$ defines a partition
	$$ M = M_1 \dot{\cup} M_2 $$
	into two sets of equal size. In particular, we have $\mid M \mid \equiv 0 \, (2)$.
\end{lemma}

\begin{proof}
	From $1 \leq l \leq \lfloor \frac{k-1}{2} \rfloor$ it follows that $k-l > l$ and $2l < k < 2(k-l)$. Thus, $ M_1 \cap M_2 = \emptyset $ and $ M_1 \cup M_2 \subseteq M$.
	
	On the other hand let $\lfloor \frac{k-1}{2} \rfloor < j < k-\lfloor \frac{k-1}{2} \rfloor$. If $k$ is odd, then $\frac{k-1}{2} < j < \frac{k+1}{2}$ which is impossible for $j \in \mathbb{N}$. If $k$ is even, then $\frac{k}{2}-1 < j < \frac{k}{2}+1$, hence it follows that $j=k$ which is excluded in the definition of $M$. Altogether $ M_1 \cup M_2 = M$	holds and 
	$$|M| = 2 \left \lfloor \frac{k-1}{2} \right \rfloor = \left \{ \begin{array}{ll} k-1 & \textrm{ if } k \textrm{ is odd} \\ k-2 & \textrm{ else.} \end{array} \right .$$
\end{proof}

\begin{lemma}
	\label{lem:typeB}
	The complex 
	$$ S_{l,2k} := \{  (l : l : 2(k-l)), (k-l : k-l : 2l) \}, $$
	$1 \leq l \leq \lfloor \frac{k-1}{2} \rfloor $, is a disjoint union of $\frac{k}{3}$ copies of $\partial \beta^3$ if $3 \mid k$ and $l = \frac{k}{3}$ and a surface of Euler characteristic $0$ otherwise.
\end{lemma}

\begin{proof}
	We prove that $S_{l,2k}$ is a surface by looking at the link of vertex $0$: 
	
	\begin{center}
		\includegraphics[width=0.38\textwidth]{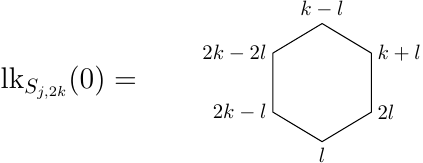}
	\end{center}
	
	where $2l = k-l$ and $2k-2l = k+l$ if and only if $l = \frac{k}{3}$. Thus, $\operatorname{lk}_{S_{l,j,2k}} (0)$ is either the boundary of a hexagon or, in the case $l = \frac{k}{3}$, the boundary of a quadrilateral and $S_{l,2k}$ is a surface.
	
	Furthermore, if  $l \neq \frac{k}{3}$ the surface $S_{l,2k}$ is a union of two difference cycles of full length and by Lemma \ref{lem:euler} we have $\chi (S_{l,2k}) = 0$. If $l = \frac{k}{3}$, $(\frac{k}{3} : \frac{k}{3} : \frac{k}{3})$ is of length $\frac{2k}{3}$ and it follows that
	$$ \chi (S_{\frac{k}{3},2k}) = 2k - \frac{8}{2}k + \frac{8}{3}k = \frac{2}{3} k.$$
	By a calculation analogue to the one in the proof of Lemma \ref{lem:typeA}, one obtains that $S_{\frac{k}{3},2k}$ consists of $\operatorname{gcd}(l,2k)=\frac{k}{3}$ isomorphic connected components of type $\{ 3,4 \}$. Hence, $S_{\frac{k}{3},2k}$ is a disjoint union of $\frac{k}{3}$ copies of $\partial \beta^3 $.
\end{proof}

Subsection 3.1 of \cite{Lutz09EquivdCovTrigsSurf} by Lutz contains a series of transitive tori and Klein bottles with $n = 8 + 2m$ vertices, $m \geq 0$. The series is given by
\begin{equation*}
	\label{thm:lutz}
	A_6 (n) := \{ ( 1 : 1 : (n-2)), (2 : (\frac{n}{2} -1) : (\frac{n}{2} -1) \}.
\end{equation*}
$A_6 (n)$ is a torus for $m$ even and a Klein bottle for $m$ odd.

\begin{lemma}
	\label{lem:gcd}
	Let $k \geq 3$, $1 \leq l \leq \lfloor \frac{k-1}{2} \rfloor $, $l \neq \frac{k}{3}$ and $n:= \operatorname{gcd}(l,k)$. Then $S_{l,2k}$ is isomorphic to $n$ copies of $A_6 ( \frac{2k}{n} )$.
\end{lemma}

\begin{proof}
	Since $n=\operatorname{gcd}(l,k)=\operatorname{gcd}(l,k-l)$, we have $n = \min \{ \operatorname{gcd}(l,2(k-l)), \operatorname{gcd}(2l,k-l)\}$ and either $\frac{l}{n} \in \mathbb{Z}_{\frac{2k}{n}}^{\times}$ or $\frac{k-l}{n} \in \mathbb{Z}_{\frac{2k}{n}}^{\times}$ holds. It follows by mulitplying with $l$ or $k-l$ that $A_6 (\frac{2k}{n})$ is isomorphic to $S_{\frac{l}{n},\frac{2k}{n}} = S_{\frac{k-l}{n},\frac{2k}{n}}$. The monomorphism
	
	$$ \mathbb{Z}_{\frac{2k}{n}} \to \mathbb{Z}_{2k} \quad j \mapsto (lj \! \mod 2k) $$
	
	represents a relabeling of $S_{\frac{l}{n},\frac{2k}{n}}$ and a small computation shows that the relabeled complex is equal to the connected component of $S_{l,2k}$ containing $0$. By a shift of the vertex labels we see that all other connected components of $S_{l,2k}$ are isomoprhic to the one containing $0$ what states the result. 
\end{proof}

Let us now come to the proof of Theorem \ref{thm:main}.

\begin{proof}
	Lemma \ref{lem:partBetaK} and Lemma \ref{lem:BSeriesPartition} describe $\operatorname{skel}_2 (\beta^k)$ in terms of $2$ series of pairs of difference cycles 
	$$ \{(l : j : 2k-l-j) , (l : 2k-l-j : j)\} \textrm{ and } \{(l : l : 2(k-l)) , (k-l : k-l : 2l)\}$$
	for certain parameters $j$ and $l$. Lemma \ref{lem:typeA} determines the topological type of the first and Theorem \ref{thm:lutz} together with Lemma \ref{lem:typeB} and \ref{lem:gcd} determines the type of the second series.
	
	Since $|\operatorname{skel}_2 (\beta^k)| = { 2k \choose 3} - k(2k-2)$ and
	for $k \not\equiv 0 \, (3)$ all surfaces have exactly $4k$ triangles, we get an overall number of $\frac{(k-1)(k-2)}{3}$ surfaces. If $k \equiv 0 \, (3)$, all surfaces but one have $4k$ triangles, the last one has $\frac{8k}{3}$ triangles. Altogether this implies that there are $\frac{k(k-3)}{3}$ surfaces of Euler characteristic $0$ and $\frac{k}{3}$ copies of $\partial \beta^3$.
\end{proof}

Table \ref{tab:listBeta} shows the decomposition of $\operatorname{skel}_2 (\beta^k)$ for $3 \leq k \leq 10$. The table was computed using the \texttt{GAP} package \texttt{simpcomp} \cite{simpcomp}. For a complete list of the decomposition for $k \leq 100$ see \cite{Spreer10SupplMatToBetaK}.

\section{The decomposition of $\operatorname{skel}_2 (\Delta^{k-1})$}
\label{sec:deltak}

First, note that $\operatorname{skel}_2 (\Delta^{k-1})$, $k \geq 3$, equals the set of all triangles on $k$ vertices. By looking at its vertex links we can see that in the case that $k$ is an even number the complex $ \{ (l:\frac{k}{2}-l:\frac{k}{2}) \}$ cannot be part of a triangulated surface for any $0 < l <  \frac{k}{2}$. Thus, the decomposition of $\operatorname{skel}_2 (\beta^{k})$ cannot be extended to a decomposition of $\operatorname{skel}_2 (\Delta^{2k-1})$ in an obvious manner. However, following Theorem \ref{thm:deltak}, in the case that $k$ is neither even nor divisible by $3$ the situation is different.

Again, we will first prove some lemma before we will come to the actual proof of the theorem.

\begin{lemma}
	\label{lem:mb}
	The complex $M_{l,k}$ with $k \geq 5$, $k \not \equiv 0 \, (3)$ and $k \not \equiv 0 \, (4)$ is a triangulation of $n:=\operatorname{gcd} (l,k)$ cylinders $[ 0,1 ] \times \partial \Delta^2$ if $\frac{k}{n}$ is even and of $n$ M\"obius strips if $\frac{k}{n}$ is odd.
\end{lemma}

\begin{proof}
	We first look at 
	$$ M_{1,k} = \{ \langle 0,1,2 \rangle , \langle 1,2,3 \rangle , \ldots , \langle k-2,k-1,0 \rangle , \langle k-1,0,1 \rangle \} $$
	for $k \geq 5$ (see Figure \ref{fig:mb}). Every triangle has exactly two neighbors. Thus, the alternating sum
	$$ + \langle 0,1,2 \rangle - \langle 1,2,3 \rangle + \ldots - + (-1)^{k-1} \langle k-1,0,1 \rangle$$
	induces an orientation if and only if $k$ is even and for any $l \in \mathbb{Z}^{\times}_k$ the complex $M_{l,k}$ is a cylinder if $k$ is even and a M\"obius strip if $k$ is odd. Now suppose that $n=\operatorname{gcd} (l,k) > 1$. Since $k \not \equiv 0 \, (3)$ and $k \not \equiv 0 \, (4)$ we have $\frac{k}{n} \geq 5$ and by a relabeling we see that the connected components of $M_{l,k}$ are combinatorially isomorphic to $M_{\frac{l}{n},\frac{k}{n}} \cong M_{1,\frac{k}{n}}$.	
\end{proof}

\begin{figure}[h]
	\centering
	\includegraphics[width=0.4\textwidth]{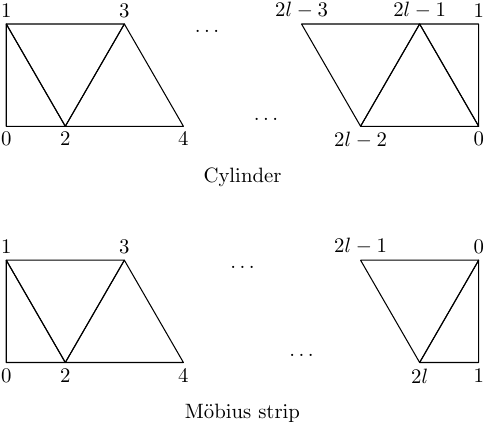}
	\caption{The cylinder $(1:1:2l-2)$ and the M\"obius strip $(1:1:2l-3)$. The vertical boundary components ($\langle 0,1 \rangle$)are identified.}
	\label{fig:mb}
\end{figure}

\begin{remark}
	If $k \equiv 0 \, (4)$ the connected components of  $M_{\frac{k}{4},k} = \{ ( \frac{k}{4} : \frac{k}{4} : \frac{k}{2} ) \}$ are equivalent to $\{ ( 1 : 1 : 2 ) \}$, the boundary of $\Delta^3$. If $k \equiv 0 \, (3)$, then $M_{\frac{k}{3},k}$ is a collection of disjoint triangles.
\end{remark}

\begin{lemma}
	\label{lem:tori}
	The complex $S_{l,j,k}$, $0 < l < j < k$, $k \not \equiv 0 \, (2)$, is a triangulation of $m:=\operatorname{gcd} (l,j,k)$ connected components of isomorphic tori on $\frac{k}{m}$ vertices.
\end{lemma}

\begin{proof}
	The link of vertex $0$ equals 
	
	\begin{center}
		\includegraphics[width=0.38\textwidth]{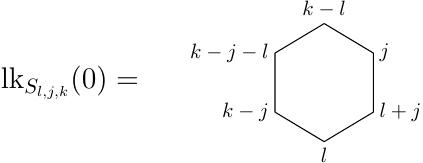}
	\end{center}
	
	(cf. proof of Lemma \ref{lem:typeA}). Since $0 < l < j < k-l-j$ and $k \not \equiv 0 \, (2)$ the link is the boundary of a hexagon, $\frac{k}{m} \geq 7$ and $S_{l,j,k}$ is a surface. By Lemma \ref{lem:euler}, the complex  $S_{l,j,k}$ is of Euler characteristic $0$. The proof of the orientability and the number of connected components is analogue to the one given in the proof of Lemma \ref{lem:typeA}. It follows that 
	$$ S_{l,j,k} \cong \{ 1, \ldots , m \} \times \mathbb{T}^2. $$
\end{proof}

Together with Lemma \ref{lem:mb} and Lemma \ref{lem:tori} in order to proof Theorem \ref{thm:deltak}, it suffices to show that the two series presented above contain all triangles of $\Delta^{k-1}$.

\begin{proof}
	Let $\langle a,b,c \rangle \in \operatorname{skel}_2 (\Delta^{k-1})$, $a<b<c$. Then $\langle a,b,c \rangle \in (b-a : c-b : k - (c-a) )$. Now if $b-a$, $c-b$ and $k-(c-a)$ are pairwise distinct, we have
	\small
	\begin{itemize}
		\item $\langle a,b,c \rangle \in S_{b-a,c-b,k} = S_{b-a,k-(c-a),k}$ if $b-a < c-b,k-(c-a)$,
		\item $\langle a,b,c \rangle \in S_{c-b,b-a,k} = S_{c-b,k-(c-a),k}$ if $c-b < b-a,k-(c-a)$ or
		\item $\langle a,b,c \rangle \in S_{k-(c-a),b-a,k} = S_{k-(c-a),c-b,k}$ if $k-(c-a) < c-b,b-a$.
	\end{itemize}
	\normalsize
	If on the other hand at least two of the entries are equal, then $(b-a : c-b : k - (c-a) ) = (l : l : k-2l )$ for $1 \leq l \leq \frac{k-1}{2}$. Thus, the union of all M\"obius strips $M_{l,k}$ and collections of tori $S_{l,j,k}$ equals the full $2$-skeleton of the $k$-simplex $\operatorname{skel}_2 (\Delta^{k-1})$.
	
	Table \ref{tab:listDelta} shows the decomposition of $\operatorname{skel}_2 (\Delta^{k-1})$ for $k \in \{ 5,7,11,13,35 \}$. For a complete list of the decomposition of $\operatorname{skel}_2 (\beta^k)$ and $\operatorname{skel}_2 (\Delta^{k-1})$ for $k \leq 100$ see \cite{Spreer10SupplMatToBetaK}.
\end{proof}

\bigskip
\scriptsize
\begin{center}
	\begin{longtable}{|l|l|l|l@{}l@{}l|}
		\caption{The decomposition of the $2$-skeleton of $\beta^k$ ($k \leq 10$) into transitive surfaces.} \label{tab:listBeta} \\
		
		\hline
		\multicolumn{1}{|l|}{$k$} & \multicolumn{1}{l|}{$f_2 (\beta^k)$} & \multicolumn{1}{l|}{\textbf{topological type}} & \multicolumn{3}{l|}{\textbf{difference cycles}} \\ \hline 
		\endfirsthead

		\multicolumn{6}{l}%
		{ \tablename\ \thetable{} -- continued from previous page} \\
		\hline \multicolumn{1}{|l|}{$k$} & \multicolumn{1}{l|}{$f_2 (\beta^k)$} & \multicolumn{1}{l|}{\textbf{topological type}} & \multicolumn{3}{l|}{\textbf{difference cycles}} \\ \hline 
		\endhead

		\hline \multicolumn{6}{r}{{continued on next page --}} \\
		\endfoot

		\hline \hline
		\endlastfoot
$3$	&	$8$&$\mathbb{S}^2$&$\{(1\!:\!1\!:\!4),(2\!:\!2\!:\!2)\}$&& \\
\hline
\hline
$4$&$32$&$\mathbb{T}^2$&$\{(1\!:\!2\!:\!5),(1\!:\!5\!:\!2)\},$&$\{(1\!:\!1\!:\!6),(3\!:\!3\!:\!2)\}$& \\
\hline
\hline
$5$&$80$&$\mathbb{T}^2$&$\{(1\!:\!2\!:\!7),(1\!:\!7\!:\!2)\},$&$\{(1\!:\!3\!:\!6),(1\!:\!6\!:\!3)\}$& \\
\hline
&&$\mathbb{K}^2$&$\{(1\!:\!1\!:\!8),(4\!:\!4\!:\!2)\},$&$\{(2\!:\!2\!:\!6),(3\!:\!3\!:\!4)\}$& \\
\hline
\hline
$6$	&	$160$&$\{ 1 , 2 \} \times \mathbb{S}^2$&$\{(2\!:\!2\!:\!8),(4\!:\!4\!:\!4)\}$&& \\
\hline
&&$\mathbb{T}^2$&$\{(1\!:\!2\!:\!9),(1\!:\!9\!:\!2)\},$&$\{(1\!:\!3\!:\!8),(1\!:\!8\!:\!3)\},$&$\{(1\!:\!4\!:\!7),(1\!:\!7\!:\!4)\},$\\
 &&&$\{(2\!:\!3\!:\!7),(2\!:\!7\!:\!3)\},$&$\{(3\!:\!4\!:\!5),(3\!:\!5\!:\!4)\},$&$\{(1\!:\!1\!:\!10),(5\!:\!5\!:\!2)\}$ \\
\hline
\hline
$7$&$280$&$\mathbb{T}^2$&$\{(1\!:\!2\!:\!11),(1\!:\!11\!:\!2)\},$&$\{(1\!:\!3\!:\!10),(1\!:\!10\!:\!3)\},$&$\{(1\!:\!4\!:\!9),(1\!:\!9\!:\!4)\},$\\
 &&&$\{(1\!:\!5\!:\!8),(1\!:\!8\!:\!5)\},$&$\{(2\!:\!3\!:\!9),(2\!:\!9\!:\!3)\},$&$\{(3\!:\!5\!:\!6),(3\!:\!6\!:\!5)\}$ \\
\hline
&&$\{ 1 , 2 \} \times \mathbb{T}^2$&$\{(2\!:\!4\!:\!8),(2\!:\!8\!:\!4)\}$&& \\
\hline
&&$\mathbb{K}^2$&$\{(1\!:\!1\!:\!12),(6\!:\!6\!:\!2)\},$&$\{(2\!:\!2\!:\!10),(5\!:\!5\!:\!4)\},$&$\{(3\!:\!3\!:\!8),(4\!:\!4\!:\!6)\}$ \\
\hline
\hline
$8$&$448$&$\mathbb{T}^2$&$\{(1\!:\!2\!:\!13),(1\!:\!13\!:\!2)\},$&$\{(1\!:\!3\!:\!12),(1\!:\!12\!:\!3)\},$&$\{(1\!:\!4\!:\!11),(1\!:\!11\!:\!4)\},$\\
 &&&$\{(1\!:\!5\!:\!10),(1\!:\!10\!:\!5)\},$&$\{(1\!:\!6\!:\!9),(1\!:\!9\!:\!6)\},$&$\{(2\!:\!3\!:\!11),(2\!:\!11\!:\!3)\},$\\
 &&&$\{(2\!:\!5\!:\!9),(2\!:\!9\!:\!5)\},$&$\{(3\!:\!4\!:\!9),(3\!:\!9\!:\!4)\},$&$\{(3\!:\!6\!:\!7),(3\!:\!7\!:\!6)\},$\\
 &&&$\{(4\!:\!5\!:\!7),(4\!:\!7\!:\!5)\},$&$\{(1\!:\!1\!:\!14),(7\!:\!7\!:\!2)\},$&$\{(3\!:\!3\!:\!10),(5\!:\!5\!:\!6)\}$ \\
\hline
&&$\{ 1 , 2 \} \times \mathbb{T}^2$&$\{(2\!:\!4\!:\!10),(2\!:\!10\!:\!4)\},$&$\{(2\!:\!2\!:\!12),(6\!:\!6\!:\!4)\}$& \\
\hline
\hline
$9$	&	$672$&$\{ 1 , 2, 3 \} \times \mathbb{S}^2$&$\{(3\!:\!3\!:\!12),(6\!:\!6\!:\!6)\}$&& \\
\hline
&&$\mathbb{T}^2$&$\{(1\!:\!2\!:\!15),(1\!:\!15\!:\!2)\},$&$\{(1\!:\!3\!:\!14),(1\!:\!14\!:\!3)\},$&$\{(1\!:\!4\!:\!13),(1\!:\!13\!:\!4)\},$\\
 &&&$\{(1\!:\!5\!:\!12),(1\!:\!12\!:\!5)\},$&$\{(1\!:\!6\!:\!11),(1\!:\!11\!:\!6)\},$&$\{(1\!:\!7\!:\!10),(1\!:\!10\!:\!7)\},$\\
 &&&$\{(2\!:\!3\!:\!13),(2\!:\!13\!:\!3)\},$&$\{(2\!:\!5\!:\!11),(2\!:\!11\!:\!5)\},$&$\{(3\!:\!4\!:\!11),(3\!:\!11\!:\!4)\},$\\
 &&&$\{(3\!:\!5\!:\!10),(3\!:\!10\!:\!5)\},$&$\{(3\!:\!7\!:\!8),(3\!:\!8\!:\!7)\},$&$\{(5\!:\!6\!:\!7),(5\!:\!7\!:\!6)\}$ \\
\hline
&&$\{ 1 , 2 \} \times \mathbb{T}^2$&$\{(2\!:\!4\!:\!12),(2\!:\!12\!:\!4)\},$&$\{(2\!:\!6\!:\!10),(2\!:\!10\!:\!6)\},$&$\{(4\!:\!6\!:\!8),(4\!:\!8\!:\!6)\}$ \\
\hline
&&$\mathbb{K}^2$&$\{(1\!:\!1\!:\!16),(8\!:\!8\!:\!2)\},$&$\{(2\!:\!2\!:\!14),(7\!:\!7\!:\!4)\},$&$\{(4\!:\!4\!:\!10),(5\!:\!5\!:\!8)\}$ \\
\hline
\hline
$10$&$960$&$\mathbb{T}^2$&$\{(1\!:\!2\!:\!17),(1\!:\!17\!:\!2)\},$&$\{(1\!:\!3\!:\!16),(1\!:\!16\!:\!3)\},$&$\{(1\!:\!4\!:\!15),(1\!:\!15\!:\!4)\},$\\
 &&&$\{(1\!:\!5\!:\!14),(1\!:\!14\!:\!5)\},$&$\{(1\!:\!6\!:\!13),(1\!:\!13\!:\!6)\},$&$\{(1\!:\!7\!:\!12),(1\!:\!12\!:\!7)\},$\\
 &&&$\{(1\!:\!8\!:\!11),(1\!:\!11\!:\!8)\},$&$\{(2\!:\!3\!:\!15),(2\!:\!15\!:\!3)\},$&$\{(2\!:\!5\!:\!13),(2\!:\!13\!:\!5)\},$\\
 &&&$\{(2\!:\!7\!:\!11),(2\!:\!11\!:\!7)\},$&$\{(3\!:\!4\!:\!13),(3\!:\!13\!:\!4)\},$&$\{(3\!:\!5\!:\!12),(3\!:\!12\!:\!5)\},$\\
 &&&$\{(3\!:\!6\!:\!11),(3\!:\!11\!:\!6)\},$&$\{(3\!:\!8\!:\!9),(3\!:\!9\!:\!8)\},$&$\{(4\!:\!5\!:\!11),(4\!:\!11\!:\!5)\},$\\
 &&&$\{(4\!:\!7\!:\!9),(4\!:\!9\!:\!7)\},$&$\{(5\!:\!6\!:\!9),(5\!:\!9\!:\!6)\},$&$\{(5\!:\!7\!:\!8),(5\!:\!8\!:\!7)\},$\\
 &&&$\{(1\!:\!1\!:\!18),(9\!:\!9\!:\!2)\},$&$\{(3\!:\!3\!:\!14),(7\!:\!7\!:\!6)\}$& \\
\hline
&&$\{ 1 , 2 \} \times \mathbb{T}^2$&$\{(2\!:\!4\!:\!14),(2\!:\!14\!:\!4)\},$&$\{(2\!:\!6\!:\!12),(2\!:\!12\!:\!6)\}$& \\
\hline
&&$\{ 1 , 2 \} \times \mathbb{K}^2$&$\{(2\!:\!2\!:\!16),(8\!:\!8\!:\!4)\},$&$\{(4\!:\!4\!:\!12),(6\!:\!6\!:\!8)\}$& \\
	\end{longtable}
\end{center}

\pagebreak
\scriptsize
\begin{center}
	\begin{longtable}{|l|l|l@{}l@{}l|}
		\caption{The decomposition of the $2$-skeleton of $\Delta^{k-1}$ ($k \in \{ 5,7,11,13,35 \} $) by topological types.} \label{tab:listDelta} \\
		
		\hline \multicolumn{1}{|l|}{$k$} & \multicolumn{1}{l|}{\textbf{topological type}} & \multicolumn{3}{l|}{\textbf{difference cycles}} \\ \hline 
		\endfirsthead

		\multicolumn{5}{l}%
		{ \tablename\ \thetable{} -- continued from previous page} \\
		\hline \multicolumn{1}{|l|}{$k$} & \multicolumn{1}{l|}{\textbf{topological type}} & \multicolumn{3}{l|}{\textbf{difference cycles}} \\ \hline 
		\endhead

		\hline \multicolumn{5}{r}{{continued on next page --}} \\
		\endfoot

		\hline \hline
		\endlastfoot
$5$&$\mathbb{M}^2$&$\{ ( 1\!:\!1\!:\!3 ) \},$&$\{ ( 2\!:\!2\!:\!1 ) \}$&\\ 
\hline 
\hline 
$7$&$\mathbb{M}^2$&$\{ ( 1\!:\!1\!:\!5 ) \},$&$\{ ( 2\!:\!2\!:\!3 ) \},$&$\{ ( 3\!:\!3\!:\!1 ) \}$\\ 
\hline 
&$\mathbb{T}^2$&$\{ ( 1\!:\!2\!:\!4 ), ( 1\!:\!4\!:\!2 ) \}$&&\\ 
\hline 
\hline 
$11$&$\mathbb{M}^2$&$\{ ( 1\!:\!1\!:\!9 ) \},$&$\{ ( 2\!:\!2\!:\!7 ) \},$&$\{ ( 3\!:\!3\!:\!5 ) \},$\\ 
&&$\{ ( 4\!:\!4\!:\!3 ) \},$&$\{ ( 5\!:\!5\!:\!1 ) \}$&\\ 
\hline 
&$\mathbb{T}^2$&$\{ ( 1\!:\!2\!:\!8 ), ( 1\!:\!8\!:\!2 ) \},$&$\{ ( 1\!:\!3\!:\!7 ), ( 1\!:\!7\!:\!3 ) \},$&$\{ ( 1\!:\!4\!:\!6 ), ( 1\!:\!6\!:\!4 ) \},$\\ 
&&$\{ ( 2\!:\!3\!:\!6 ), ( 2\!:\!6\!:\!3 ) \},$&$\{ ( 2\!:\!4\!:\!5 ), ( 2\!:\!5\!:\!4 ) \}$&\\ 
\hline 
\hline 
$13$&$\mathbb{M}^2$&$\{ ( 1\!:\!1\!:\!11 ) \},$&$\{ ( 2\!:\!2\!:\!9 ) \},$&$\{ ( 3\!:\!3\!:\!7 ) \},$\\ 
&&$\{ ( 4\!:\!4\!:\!5 ) \},$&$\{ ( 5\!:\!5\!:\!3 ) \},$&$\{ ( 6\!:\!6\!:\!1 ) \}$\\ 
\hline 
&$\mathbb{T}^2$&$\{ ( 1\!:\!2\!:\!10 ), ( 1\!:\!10\!:\!2 ) \},$&$\{ ( 1\!:\!3\!:\!9 ), ( 1\!:\!9\!:\!3 ) \},$&$\{ ( 1\!:\!4\!:\!8 ), ( 1\!:\!8\!:\!4 ) \},$\\ 
&&$\{ ( 1\!:\!5\!:\!7 ), ( 1\!:\!7\!:\!5 ) \},$&$\{ ( 2\!:\!3\!:\!8 ), ( 2\!:\!8\!:\!3 ) \},$&$\{ ( 2\!:\!4\!:\!7 ), ( 2\!:\!7\!:\!4 ) \},$\\ 
&&$\{ ( 2\!:\!5\!:\!6 ), ( 2\!:\!6\!:\!5 ) \},$&$\{ ( 3\!:\!4\!:\!6 ), ( 3\!:\!6\!:\!4 ) \}$&\\ 
\hline 
\hline 
$35$&$\mathbb{M}^2$&$\{ ( 1\!:\!1\!:\!33 ) \},$&$\{ ( 2\!:\!2\!:\!31 ) \},$&$\{ ( 3\!:\!3\!:\!29 ) \},$\\ 
&&$\{ ( 4\!:\!4\!:\!27 ) \},$&$\{ ( 6\!:\!6\!:\!23 ) \},$&$\{ ( 8\!:\!8\!:\!19 ) \},$\\ 
&&$\{ ( 9\!:\!9\!:\!17 ) \},$&$\{ ( 11\!:\!11\!:\!13 ) \},$&$\{ ( 12\!:\!12\!:\!11 ) \},$\\ 
&&$\{ ( 13\!:\!13\!:\!9 ) \},$&$\{ ( 16\!:\!16\!:\!3 ) \},$&$\{ ( 17\!:\!17\!:\!1 ) \}$\\ 
\hline 
&$\{ 1 , \ldots , 5 \} \times \mathbb{M}^2$&$\{ ( 5\!:\!5\!:\!25 ) \},$&$\{ ( 10\!:\!10\!:\!15 ) \},$&$\{ ( 15\!:\!15\!:\!5 ) \}$\\ 
\hline 
&$\{ 1 , \ldots , 7 \} \times \mathbb{M}^2$&$\{ ( 7\!:\!7\!:\!21 ) \},$&$\{ ( 14\!:\!14\!:\!7 ) \}$&\\ 
\hline 
&$\mathbb{T}^2$&$\{ ( 1\!:\!2\!:\!32 ), ( 1\!:\!32\!:\!2 ) \},$&$\{ ( 1\!:\!3\!:\!31 ), ( 1\!:\!31\!:\!3 ) \},$&$\{ ( 1\!:\!4\!:\!30 ), ( 1\!:\!30\!:\!4 ) \},$\\ 
&&$\{ ( 1\!:\!5\!:\!29 ), ( 1\!:\!29\!:\!5 ) \},$&$\{ ( 1\!:\!6\!:\!28 ), ( 1\!:\!28\!:\!6 ) \},$&$\{ ( 1\!:\!7\!:\!27 ), ( 1\!:\!27\!:\!7 ) \},$\\ 
&&$\{ ( 1\!:\!8\!:\!26 ), ( 1\!:\!26\!:\!8 ) \},$&$\{ ( 1\!:\!9\!:\!25 ), ( 1\!:\!25\!:\!9 ) \},$&$\{ ( 1\!:\!10\!:\!24 ), ( 1\!:\!24\!:\!10 ) \},$\\ 
&&$\{ ( 1\!:\!11\!:\!23 ), ( 1\!:\!23\!:\!11 ) \},$&$\{ ( 1\!:\!12\!:\!22 ), ( 1\!:\!22\!:\!12 ) \},$&$\{ ( 1\!:\!13\!:\!21 ), ( 1\!:\!21\!:\!13 ) \},$\\ 
&&$\{ ( 1\!:\!14\!:\!20 ), ( 1\!:\!20\!:\!14 ) \},$&$\{ ( 1\!:\!15\!:\!19 ), ( 1\!:\!19\!:\!15 ) \},$&$\{ ( 1\!:\!16\!:\!18 ), ( 1\!:\!18\!:\!16 ) \},$\\ 
&&$\{ ( 2\!:\!3\!:\!30 ), ( 2\!:\!30\!:\!3 ) \},$&$\{ ( 2\!:\!4\!:\!29 ), ( 2\!:\!29\!:\!4 ) \},$&$\{ ( 2\!:\!5\!:\!28 ), ( 2\!:\!28\!:\!5 ) \},$\\ 
&&$\{ ( 2\!:\!6\!:\!27 ), ( 2\!:\!27\!:\!6 ) \},$&$\{ ( 2\!:\!7\!:\!26 ), ( 2\!:\!26\!:\!7 ) \},$&$\{ ( 2\!:\!8\!:\!25 ), ( 2\!:\!25\!:\!8 ) \},$\\ 
&&$\{ ( 2\!:\!9\!:\!24 ), ( 2\!:\!24\!:\!9 ) \},$&$\{ ( 2\!:\!10\!:\!23 ), ( 2\!:\!23\!:\!10 ) \},$&$\{ ( 2\!:\!11\!:\!22 ), ( 2\!:\!22\!:\!11 ) \},$\\ 
&&$\{ ( 2\!:\!12\!:\!21 ), ( 2\!:\!21\!:\!12 ) \},$&$\{ ( 2\!:\!13\!:\!20 ), ( 2\!:\!20\!:\!13 ) \},$&$\{ ( 2\!:\!14\!:\!19 ), ( 2\!:\!19\!:\!14 ) \},$\\ 
&&$\{ ( 2\!:\!15\!:\!18 ), ( 2\!:\!18\!:\!15 ) \},$&$\{ ( 2\!:\!16\!:\!17 ), ( 2\!:\!17\!:\!16 ) \},$&$\{ ( 3\!:\!4\!:\!28 ), ( 3\!:\!28\!:\!4 ) \},$\\ 
&&$\{ ( 3\!:\!5\!:\!27 ), ( 3\!:\!27\!:\!5 ) \},$&$\{ ( 3\!:\!6\!:\!26 ), ( 3\!:\!26\!:\!6 ) \},$&$\{ ( 3\!:\!7\!:\!25 ), ( 3\!:\!25\!:\!7 ) \},$\\ 
&&$\{ ( 3\!:\!8\!:\!24 ), ( 3\!:\!24\!:\!8 ) \},$&$\{ ( 3\!:\!9\!:\!23 ), ( 3\!:\!23\!:\!9 ) \},$&$\{ ( 3\!:\!10\!:\!22 ), ( 3\!:\!22\!:\!10 ) \},$\\ 
&&$\{ ( 3\!:\!11\!:\!21 ), ( 3\!:\!21\!:\!11 ) \},$&$\{ ( 3\!:\!12\!:\!20 ), ( 3\!:\!20\!:\!12 ) \},$&$\{ ( 3\!:\!13\!:\!19 ), ( 3\!:\!19\!:\!13 ) \},$\\ 
&&$\{ ( 3\!:\!14\!:\!18 ), ( 3\!:\!18\!:\!14 ) \},$&$\{ ( 3\!:\!15\!:\!17 ), ( 3\!:\!17\!:\!15 ) \},$&$\{ ( 4\!:\!5\!:\!26 ), ( 4\!:\!26\!:\!5 ) \},$\\ 
&&$\{ ( 4\!:\!6\!:\!25 ), ( 4\!:\!25\!:\!6 ) \},$&$\{ ( 4\!:\!7\!:\!24 ), ( 4\!:\!24\!:\!7 ) \},$&$\{ ( 4\!:\!8\!:\!23 ), ( 4\!:\!23\!:\!8 ) \},$\\ 
&&$\{ ( 4\!:\!9\!:\!22 ), ( 4\!:\!22\!:\!9 ) \},$&$\{ ( 4\!:\!10\!:\!21 ), ( 4\!:\!21\!:\!10 ) \},$&$\{ ( 4\!:\!11\!:\!20 ), ( 4\!:\!20\!:\!11 ) \},$\\ 
&&$\{ ( 4\!:\!12\!:\!19 ), ( 4\!:\!19\!:\!12 ) \},$&$\{ ( 4\!:\!13\!:\!18 ), ( 4\!:\!18\!:\!13 ) \},$&$\{ ( 4\!:\!14\!:\!17 ), ( 4\!:\!17\!:\!14 ) \},$\\ 
&&$\{ ( 4\!:\!15\!:\!16 ), ( 4\!:\!16\!:\!15 ) \},$&$\{ ( 5\!:\!6\!:\!24 ), ( 5\!:\!24\!:\!6 ) \},$&$\{ ( 5\!:\!7\!:\!23 ), ( 5\!:\!23\!:\!7 ) \},$\\ 
&&$\{ ( 5\!:\!8\!:\!22 ), ( 5\!:\!22\!:\!8 ) \},$&$\{ ( 5\!:\!9\!:\!21 ), ( 5\!:\!21\!:\!9 ) \},$&$\{ ( 5\!:\!11\!:\!19 ), ( 5\!:\!19\!:\!11 ) \},$\\ 
&&$\{ ( 5\!:\!12\!:\!18 ), ( 5\!:\!18\!:\!12 ) \},$&$\{ ( 5\!:\!13\!:\!17 ), ( 5\!:\!17\!:\!13 ) \},$&$\{ ( 5\!:\!14\!:\!16 ), ( 5\!:\!16\!:\!14 ) \},$\\ 
&&$\{ ( 6\!:\!7\!:\!22 ), ( 6\!:\!22\!:\!7 ) \},$&$\{ ( 6\!:\!8\!:\!21 ), ( 6\!:\!21\!:\!8 ) \},$&$\{ ( 6\!:\!9\!:\!20 ), ( 6\!:\!20\!:\!9 ) \},$\\ 
&&$\{ ( 6\!:\!10\!:\!19 ), ( 6\!:\!19\!:\!10 ) \},$&$\{ ( 6\!:\!11\!:\!18 ), ( 6\!:\!18\!:\!11 ) \},$&$\{ ( 6\!:\!12\!:\!17 ), ( 6\!:\!17\!:\!12 ) \},$\\ 
&&$\{ ( 6\!:\!13\!:\!16 ), ( 6\!:\!16\!:\!13 ) \},$&$\{ ( 6\!:\!14\!:\!15 ), ( 6\!:\!15\!:\!14 ) \},$&$\{ ( 7\!:\!8\!:\!20 ), ( 7\!:\!20\!:\!8 ) \},$\\ 
&&$\{ ( 7\!:\!9\!:\!19 ), ( 7\!:\!19\!:\!9 ) \},$&$\{ ( 7\!:\!10\!:\!18 ), ( 7\!:\!18\!:\!10 ) \},$&$\{ ( 7\!:\!11\!:\!17 ), ( 7\!:\!17\!:\!11 ) \},$\\ 
&&$\{ ( 7\!:\!12\!:\!16 ), ( 7\!:\!16\!:\!12 ) \},$&$\{ ( 7\!:\!13\!:\!15 ), ( 7\!:\!15\!:\!13 ) \},$&$\{ ( 8\!:\!9\!:\!18 ), ( 8\!:\!18\!:\!9 ) \},$\\ 
&&$\{ ( 8\!:\!10\!:\!17 ), ( 8\!:\!17\!:\!10 ) \},$&$\{ ( 8\!:\!11\!:\!16 ), ( 8\!:\!16\!:\!11 ) \},$&$\{ ( 8\!:\!12\!:\!15 ), ( 8\!:\!15\!:\!12 ) \},$\\ 
&&$\{ ( 8\!:\!13\!:\!14 ), ( 8\!:\!14\!:\!13 ) \},$&$\{ ( 9\!:\!10\!:\!16 ), ( 9\!:\!16\!:\!10 ) \},$&$\{ ( 9\!:\!11\!:\!15 ), ( 9\!:\!15\!:\!11 ) \},$\\ 
&&$\{ ( 9\!:\!12\!:\!14 ), ( 9\!:\!14\!:\!12 ) \},$&$\{ ( 10\!:\!11\!:\!14 ), ( 10\!:\!14\!:\!11 ) \},$&$\{ ( 10\!:\!12\!:\!13 ), ( 10\!:\!13\!:\!12 ) \}$\\ 
\hline 
&$\{ 1 , \ldots , 5 \} \times \mathbb{T}^2$&$\{ ( 5\!:\!10\!:\!20 ), ( 5\!:\!20\!:\!10 ) \}$&&\\ 
		\end{longtable}
\end{center}
\pagebreak
\normalsize

\addcontentsline{toc}{chapter}{Bibliography}
{\footnotesize
 \bibliographystyle{abbrv}
 \bibliography{bibliography}
}
\noindent
Institut f\"ur Geometrie und Topologie \\
Universit\"at Stuttgart \\
70550 Stuttgart \\
Germany
\end{document}